\documentclass{amsart}

\makeatletter
 \def\LaTeX{\leavevmode L\raise.42ex
   \hbox{\kern-.3em\size{\sf@size}{0pt}\selectfont A}\kern-.15em\TeX}
\makeatother

\newcommand{\BibTeX}{{\rm B\kern-.05em{\sc
i\kern-.025emb}\kern-.08em\TeX}}
\newtheorem{col}{Corollary}[section]
\newtheorem{thm}{Theorem}[section]
\newtheorem{lem}[thm]{Lemma}

\newtheorem{rem}[thm]{Remark}
\theoremstyle{defn}
\newtheorem{defn}{Definition}

\newtheorem{assump}{Assumptions}

\numberwithin{equation}{section}

\newcommand{\Ltwo}{\ell^2}

\begin{document}

\title[Weighted  sampling and interpolation on graphs ]{Weighted  sampling and weighted interpolation  on combinatorial graphs}

\author{Isaac Z. Pesenson}
\address{Department of Mathematics, Temple University,
 Philadelphia,
PA 19122}

\email{pesenson@temple.edu}

\maketitle

\begin{abstract}

For  Paley-Wiener functions on  weighted combinatorial finite or infinite  graphs   we develop a weighted sampling theory in which samples are defined as inner products with weight functions (measuring devices). Three reconstruction methods are suggested.  The first  two of them are  using language of dual  Hilbert frames and the so-called frame algorithm respectively. The third  one is using the so-called weighted  variational interpolating splines which are constructed in the setting of  combinatorial graphs. This development requires a new set of Poincar\'e-type inequalities which we prove for functions on combinatorial graphs.

\end{abstract}

\section{Introduction and main results}

\label{sect:main}

During the last decade signal processing  on graphs was developed in a number of papers, for example, in \cite {FP}, \cite{K}, \cite {Pes08a}-\cite{WA}. Many of the papers on  this list considered what can be called as a "point-wise sampling". The goal of the present article is to develop sampling on graphs which is based on weighted averages over relatively small subgraphs. The idea to use local information (other than point values) for reconstruction of bandlimited functions on graphs was already explored in   \cite{W}. However, the results and methods of \cite{W} and of our paper are very different.  We also want to mention that results  of the present paper are similar to results of our papers \cite{Pes01} and \cite{Pes04b} in which sampling by weighted average values was developed in abstract Hilbert spaces and on Riemannian manifolds.

Let $G$ denote an undirected weighted graph, with  a finite or countable number of vertices $V(G)$  and weight function $w: V(G) \times V(G) \to \mathbb{R}_0^+$. $w$ is symmetric, i.e., $w(u,v) = w(v,u)$, and $w(u,u)=0$ for all $u,v \in V(G)$. The edges of the graph are the pairs $(u,v)$ with $w(u,v) \not= 0$. 
Our assumption is that for every $v\in V(G)$ the following finiteness condition holds
\begin{equation}\label{cond:finiteness}
 w(v) = \sum_{u \in V(G)} w(u,v)<\infty.
\end{equation}
Let $\ell^{2}(G)\>\>$ denote the space of  all complex-valued functions with the inner product
$$
\left<f,g\right>=\sum_{v\in V(G)}f(v)\overline{g(v)}
$$
and  the norm
\[
\|f\|_{G}= \| f \| = \left( \sum_{v \in V(G)} |f(v)|^2 \right)^{1/2}.
\] 
\begin{defn}The  weighted gradient norm of a function $f$ on $V(G)$ is defined by 
\begin{equation}\label{gr}
 \| \nabla f \| = \left( \sum_{u, v \in V(G)} \frac{1}{2} |f(u) - f(v)|^2 w(u,v) \right)^{1/2}.
\end{equation} 
 The set of all $f: G \to \mathbb{C}$ for which the weighted gradient norm is finite will be denoted as $\mathcal{D}(\nabla)$.
\end{defn}

\begin{rem}
The factor $\frac{1}{2}$ makes up for the fact that every edge (i.e., every {\em unordered} pair $(u,v)$) enters twice in the summation. Note also that loops, i.e. edges of the type $(u,u)$, in fact do not contribute.

\end{rem}

We intend to prove Poincar\'e-type estimates involving weighted gradient norm.

In  the case of a \textit{finite} graph and $\Ltwo(G)$-space  the weighted Laplace operator $L: \Ltwo(G) \to \Ltwo(G)$  is introduced  via
\begin{equation}\label{L}
 (L f)(v) = \sum_{u \in V(G)} (f(v)-f(u)) w(v,u)~.
\end{equation}
This  graph Laplacian is a well-studied object; it is known to be a positive-semidefinite self-adjoint \textit{bounded} operator.

According to  Theorem 8.1 and Corollary 8.2 in \cite{H}  if for an \textit{infinite} graph    there exists a $C>0$ such that the degrees are uniformly bounded 
\begin{equation}\label{cond:finiteness-1}
w(v) = \sum_{u \in V(G)} w(u,v)\leq C,
\end{equation}
then operator which is defined by (\ref{L}) on functions with compact supports has a unique  positive-semidefinite self-adjoint \textit{bounded} extension $L$ which is acting according to  (\ref{L}).

In section \ref{sec:finite} we consider 
 a \textit{finite connected graph $G$ which contains more than one vertex} and a functional    $\Psi$ on $\ell^{2}(G)$ which is defined by a function $\psi\in \ell^{2}(G)$, i.e. $$
 \Psi(f)=\langle f,\psi\rangle=\sum_{v\in V(G)}f(v)\overline{\psi(v)}.
 $$
We will use notation  $\chi_{G}$  for the characteristic function:  $\chi_{G}(v)=1$ for all $v\in V(G)$. In  these notions  we prove (Theorem \ref{fund-Th-0})  that 
if $\Psi(\chi_{G})$ is not zero then for any  $f\in Ker(\Psi)$ the following inequality holds
\begin{equation}\label{fund-ineq-0}
\|f\|^{2}\leq \frac{\theta}{\lambda_{1}}\|\nabla f\|^{2},\>\>\>f\in Ker(\Psi),
\end{equation}
where $\lambda_{1}$ is the first non zero eigenvalue of the Laplacian (\ref{L}) and
\begin{equation}\label{Teta}
 \theta=\frac{|G|\|\psi\|^{2}}{|\Psi(\chi_{G} )|^{2}},
\end{equation}
where $|G|$ is cardinality  of $V(G)$.

In section \ref{sec 3} we extend this result to situations in which a cover by \textit{finite and connected } subgraphs of a \textit{ finite or infinite}  graph $G$ is given. Namely, we are working under the following assumptions.

\bigskip

\begin{assump}\label{Assump}

We assume that 
$\mathcal{S}=\{S_{j}\}_{j\in J}$ form a cover of $V(G)$
\begin{equation}\label{cover}
 \bigcup_{j\in J}S_{j}=V(G).
\end{equation}
We don't assume that the sets $S_{j}$ are disjoint but we assume that \textit{there is no any edge in $E(G)$ which belongs to two different subsets $S_{j},\>j\in J.$}

Let $L_{j}$  be the Laplacian  for the {\bf induced subgraph} $S_{j}$. In order to insure that $L_{j}$ has at least one non zero eigenvalue, we assume  that every $S_{j}\subset V(G),\>\>j\in J, $ is a \textit{finite and connected subset of vertices with more than one vertex. } The first nonzero eigenvalue of the operator  $L_{j}$ will be denoted as $\lambda_{1, j}$. 
Let $\|\nabla_{j}f_{j}\|$ be the weighted gradient for the  induced subgraph  $S_{j}$.
With every $S_{j},\>\>j\in J,$ we associate a function $\psi_{j}\in \ell^{2}(G)$ whose support is in $S_{j}$
and introduce  the  functionals $\Psi_{j}$ on $\ell^{2}(G)$ defined by these functions
$$
\Psi_{j}(f)= \langle f, \psi_{j} \rangle=\sum_{v\in V(S_{j})}f(v)\psi_{j}(v),\>\>\>\>f\in \ell^{2}(G).
$$
Notation $\chi_{j}$ will be used for characteristic function of   $S_{j}$ and  use $f_{j}$ for $f\chi_{j},\>\>f\in \ell^{2}(G)$.

\end{assump}

As usual, the induced graph $S_{j}$ has the same vertices as the set $S_{j}$ but only such edges of $E(G)$ which have both ends in $S_{j}$.

\bigskip

The two inequalities below (\ref{main-ineq-200}) and  (\ref{main-ineq-2000}) are essentially the main inequalities we prove in section \ref{sec 3}. We call them \textit{ generalized Poincar\'e-type inequalities} since they contain an estimate of a function through its smoothness. 
Namely, we show that if  
$$
\Psi_{j}(\chi_{j})=\sum_{v\in S_{j}} \psi_{j}(v)\neq 0,
$$
and
\begin{equation}\label{teta}
 \theta_{j}=\frac{|S_{j}|\|\psi_{j}\|^{2}}{|\Psi_{j}(\chi_{j} )|^{2}},
\end{equation}
then the following inequalities  hold for every $f\in \ell^{2}(G)$ and every $\epsilon>0$

\begin{equation}\label{main-ineq-200}
\|f\|^{2}\leq    (1+\epsilon)  \sum_{j\in J}\frac{\theta_{j}}{\lambda_{1,j}}         \|\nabla_{j} f_{j}\|^{2}    +
  \frac{1+\epsilon}{\epsilon}\sum_{j\in J}\frac{|S_{j}|^{2}}{|\Psi_{j}(\chi_{j})|^{2}}|\Psi_{j}(f_{j})|^{2},
\end{equation}

\begin{equation}\label{main-ineq-2000}
\|f\|^{2}\leq (1+\epsilon)\frac{\Theta_{\Xi}
}{\Lambda_{\mathcal{S}}
} \|L^{1/2}f\|^{2}+\frac{1+\epsilon}{\epsilon}\sum_{j\in J}\frac{|S_{j}|^{2}}{|\Psi_{j}(\chi_{j})|^{2}}\left|\Psi_{j}(f_{j})\right|^{2}.
\end{equation}
where $\Xi=\left(\mathcal{S},\>\{\Psi_{j}\}_{j\in J}\right)$, $\>\>\>\mathcal{S}=\{S_{j}\}_{j\in J}$,
\begin{equation}\label{Theta}
\Theta_{\Xi}
=\sup_{j\in J} \theta_{j}<\infty,
\end{equation}
where $\theta_{j}$ is  computed according to (\ref{Teta}) and 
\begin{equation}\label{Lambda}
\Lambda_{\mathcal{S}}
=\inf_{j\in J}\lambda_{1,j}>0.
\end{equation}

Note, that an important situation  occurs  in (\ref{main-ineq-200}) and (\ref{main-ineq-2000})  when $f\in \cap_{j\in J}Ker\>\Psi_{j}$.
 In this case one has 
 \begin{equation}\label{main-ineq-201}
\|f\|^{2}\leq   \sum_{j\in J}\frac{\theta_{j}}{\lambda_{1,j}}         \|\nabla_{j} f_{j}\|^{2},\>\>\>\>\>\>f\in \bigcap_{j\in J}Ker\>\Psi_{j},
\end{equation}
and 
 
\begin{equation}\label{main-ineq-2001}
\|f\|^{2}\leq\>\>\frac{\Theta_{\Xi}
}{\Lambda_{\mathcal{S}}
} \|L^{1/2}f\|^{2},\>\>\>\>\>\>f\in \bigcap_{j\in J}Ker\>\Psi_{j}.
\end{equation}
Another interesting case occurs when  for every $j\in J$ the functional $\Psi_{j}$ is a Dirac measure $\delta_{v_{j}}$ at a vertex $v_{j}\in S_{j}$. In this case the condition $f\in \cap_{j\in J}Ker\>\delta_{v_{j}}$ means that $
f(v_{j})=0,\>\>\>\>\> j\in J, 
$
and one obtains (see (\ref{local 0s}) and (\ref{global 0s}) below)

\begin{equation}\label{local}
\|f\|^{2}\leq   \sum_{j\in J}  \frac{|S_{j}|}{\lambda_{1,j}}       \|\nabla_{j} f_{j}\|^{2},\>\>\>f\in \bigcap_{j\in J}Ker\>\delta_{v_{j}},
\end{equation}
and
\begin{equation}\label{global}
\|f\|^{2}\leq    \frac{\sup_{j\in J}|S_{j}| }{\Lambda_{\mathcal{S}}
}     \| L^{1/2} f\|^{2},\>\>\>f\in \bigcap_{j\in J}Ker\>\delta_{v_{j}}.
\end{equation}
A few more interesting particular situations will be discussed at the end of section \ref{sec 3}. 
We also have similar inequalities for subgraphs  (see formulas (\ref{Many zeros-3}) and  (\ref{Many zeros-4}) below).
Let's note, that in the continuous case (see \cite{Pes00}-\cite{Pes04b})  such inequalities play an important role in the sampling and interpolation theories  on Riemannian manifolds.

\begin{rem}
It is interesting to note that if one will rearrange and mutually connect subgraphs $\{S_{j}\}_{j\in J}$ in any other way to obtain a new graph $\widetilde{G}$ then the  "local" inequalities  (\ref{main-ineq-200}), (\ref{main-ineq-201}), (\ref{local})  will stay the same but the "global" ones  (\ref{main-ineq-2000}),  (\ref{main-ineq-2001}), (\ref{global} ) will change since  they will involve a new Laplacian $\widetilde{L}$ which corresponds to $\widetilde{G}$. 
 
\end{rem}

It is worth to stress that  the   "local" inequalities  (\ref{main-ineq-200}), (\ref{main-ineq-201}), (\ref{local})  are quite informative and capture   highly irregular local structures of graphs. Indeed, in the case, say, of a Riemannian manifold a difference between two small neighborhoods $S_{i}$ and $S_{j}$ is essentially their diameter. However, in the case of a graph two different even "small" sets can have very different structures. These differences are better reflected by quantities like $\lambda_{1,j}$ and $\theta_{j}$.  

Let's also note that from the practical point of view, the averaging procedure (which corresponds the case  when $\psi_{j}$ is the characteristic function $ \chi_{j}$ of a subset  $S_{j}$) can be instrumental in reducing noise inherited into point wise measurements.

In section \ref{PWspace} we introduce Paley-Wiener spaces $PW_{\omega}$ for finite and infinite graphs. In section \ref{frames} using       inequalities  (\ref{main-ineq-200}) and (\ref{main-ineq-2000}) and their variations we develop a sampling theory for Paley-Wiener functions on finite and infinite graphs (Theorems \ref{Main-Th2} and \ref{conclusion}).  At this point for reconstruction of functions from weighted average samples   we adopt    dual  Hilbert frames and the so-called frame algorithm.

In section \ref{w-splines}  by using inequality (\ref{main-ineq-2001}) we  outline a construction of variational interpolating splines which interpolate functions using their weighted average values over subsets. It is shown that Paley-wiener functions can be reconstructed using weighted average interpolating splines when smoothness of splines goes to infinity. In section \ref{example} we illustrate some of our results using infinite graph $\mathbb{Z}$.

\section{A Poincare-type inequality for finite graphs}\label{sec:finite}

The following lemma is important for us (see \cite{Mo} for finite graphs, and \cite{FP} for infinite). 

\begin{lem}\label{grad-laplace}
 If a graph  $G$ is finite or the condition (\ref{cond:finiteness-1}) is satisfied  then  one has the equality
\begin{equation}\label{L-G}
\|L^{1/2}f\|=\|\nabla f  \|
\end{equation}
for all  $f\in\ell^{2}(G)$.

\end{lem}
\begin{proof} It is easy to verify that under  assumption (\ref{cond:finiteness-1}) the domain $\mathcal{D}(\nabla)$ coincides with $\ell^{2}(G)$. 
 Let $d(u) = w_{V(G)}(u)$.  Then we obtain 
\begin{eqnarray*}
 \langle f, L f \rangle & = &  \sum_{u \in V(G)} f(u) \overline{ \left( \sum_{v \in V(G)} \left( f(u) - f(v) \right) w(u,v) \right)} \\
 & = & \sum_{u \in V(G)} \left( |f(u)|^2 d(u) - \sum_{v  \in V(G)} f(u) \overline{f(v)} w(u,v) \right) ~.
\end{eqnarray*}
In the same way 
\begin{eqnarray*}
 \langle f, L f \rangle & = & \langle L f, f \rangle \\
 & = & \sum_{u \in V(G)} \left( |f(u)|^2 d(u) - \sum_{v  \in V(G)} \overline{f(u)} f(v) w(u,v) \right)~.
\end{eqnarray*}
Averaging these equations yields
\begin{eqnarray*}
 \langle f, L f \rangle &  = &  \sum_{u \in V(G)} \left( |f(u)|^2 d(u) - {\rm Re} \sum_{v  \in V(G)} f(u) \overline{f(v)} w(u,v) \right)  \\
& = & \frac{1}{2} \sum_{ u,v \in V(G)} |f(u)|^2 w(u,v) + |f(v)|^2 w(u,v) - 2  {\rm Re}  f(u) \overline{f(v)} w(u,v) \\
& = & \sum_{u,v  \in V(G)} \frac{1}{2} |f(v)-f(u)|^2 w(u,v)  =  \| \nabla f \|^2~.
\end{eqnarray*}
Lemma is proved.
\end{proof}

For a \textit{finite connected graph $G$ which contains more than one vertex} let   $\Psi$ be a functional on $\ell^{2}(G)$ which is defined by a function $\psi\in \ell^{2}(G)$, i.e. $$
 \Psi(f)=\langle f,\psi\rangle=\sum_{v\in V(G)}f(v)\overline{\psi(v)}.
 $$
We will use notation  $\chi_{G}$  for the characteristic function:  $\chi_{G}(v)=1$ for all $v\in V(G)$. 
Using these notions  we prove the following.
\begin{thm}\label{fund-Th-0}
Let $G$ be a finite connected graph which contains more than one vertex and $\Psi(\chi_{G})$ is not zero.  If $f\in Ker(\Psi)$ then
\begin{equation}\label{fund-ineq}
\|f\|^{2}\leq \frac{\theta}{\lambda_{1}}\|\nabla f\|^{2},\>\>\>f\in Ker(\Psi),
\end{equation}
where $\lambda_{1}$ is the first non zero eigenvalue of the Laplacian (\ref{L}) and
\begin{equation}\label{Teta}
 \theta=\frac{|G|\|\psi\|^{2}}{|\Psi(\chi_{G} )|^{2}},
\end{equation}
where $|G|$ is cardinality  of $V(G)$.
\end{thm}

  \begin{proof}

If $\lambda_{0}<\lambda_{1}\leq .... \lambda_{N-1},\>\>\>N=|G|$ is the set of eigenvalue and $\varphi_{0}, \varphi_{1},...,\varphi_{N-1}$ is a set of orthonormal eigenfunctions then
$\{c_{k}(f)=\langle f, \varphi_{k}\rangle\}$ is a set of Fourier coefficients. One has 
$$
f=\sum_{k}^{N-1}c_{k}(f)\varphi_{k}
$$
 and if $f\in Ker(\Psi)$ then
 $$
 0=\Psi(f)=\frac{1}{\sqrt{|G|}}c_{0}(f)\Psi(\chi_{G})+\sum_{k=}^{N-1}c_{k}(f)\Psi(\varphi_{k}).
 $$
 From here
 $$
 c_{0}(f)=-\frac{\sqrt{|G|}}{\Psi(\chi_{G})}\sum_{k=1}^{N-1}c_{k}(f)\Psi(\varphi_{k}),
 $$
 and then using Parseval equality and Schwartz inequality we obtain

\begin{equation}\label{f-la}
\|f\|^{2}=|c_{0}(f)|^{2}+\sum_{k=1}^{N-1}|c_{k}(f)|^{2}=\frac{|G|}{|\Psi(\chi_{G})|^{2}}\left|\sum_{k=1}^{N-1}c_{k}(f)\Psi(\varphi_{k})\right|^{2}+\sum_{k=1}^{N-1}|c_{k}(f)|^{2}\leq
$$
$$
\frac{|G|}{|\Psi(\chi_{G})|^{2}}\sum_{k=1}^{N-1}|c_{k}(f)|^{2} \sum_{k=1}^{N-1}|\Psi(\varphi_{k})|^{2}            +\sum_{k=1}^{N-1}|c_{k}(f)|^{2}.
\end{equation}
At the same time, since $\varphi_{0}=\frac{\chi_{G}}{\sqrt{|G|}}$ and $\langle \psi,\varphi_{k}\rangle=\Psi(\varphi_{k})$ we have
$$
\psi=\frac{1}{\sqrt{|G|}}\Psi( \chi_{G})\varphi_{0}+\sum_{k=1}^{N-1}\Psi(\varphi_{k})\varphi_{k},
$$
and from Parseval formula
$$
\sum_{k=1}^{N-1}|\Psi(\varphi_{k})|^{2}=\|\psi\|^{2}-\frac{|\Psi(\chi_{G})|^{2}}{|G|}.
$$
We plug the right-hand side of this formula into (\ref{f-la}) and obtain the next inequality in which  $\theta$ is given by (\ref{Teta}) 
$$
\|f\|^{2}\leq \theta\sum_{k=1}^{N-1}|c_{k}(f)|^{2}\leq \frac{\theta}{\lambda_{1}}\sum_{k=1}^{N-1}|\lambda_{k}^{1/2}c_{k}(f)|^{2}= \frac{\theta}{\lambda_{1}}\|L^{1/2}\|^{2}.
$$
To finish the proof one has to apply Lemma \ref{grad-laplace}. Theorem is proven.
\end{proof}

\begin{col}\label{weighted-Poinc}
Let $G$ be a finite connected graph which contains more than one vertex and $\Psi(\chi_{G})$ is not zero. Then one has for every $f\in     \ell^{2}(G)$
\begin{equation}
\left\|f-\frac{\Psi(f)}{\Psi(\chi_{G})}\chi_{G}\right\|^{2}\leq \frac{\theta}{\lambda_{1}}\|\nabla f\|^{2},
\end{equation}
where $\theta$ as in (\ref{Teta}).

\end{col}

The proof follows from the fact that for $g=f-\frac{\Psi(f)}{\Psi(\chi_{G})}\chi_{G}$ the following properties hold:
$$
g\in Ker\>(\Psi),\>\>\>\>\> \nabla g=\nabla f.
$$

When $\psi$ equals to the eigenfunction $\varphi_{0}$ then for the corresponding functional $\Psi_{0}$ the condition $f\in Ker(\Psi_{0}) $ is equivalent to $\langle f, \varphi_{0}\rangle=0$. It is easy to see that in this case $\theta=1$ and then (\ref{fund-ineq}) gives the following Corollary.

\begin{col}
If $\langle f, \varphi_{0}\rangle=0$ then
\begin{equation}\label{Rm-0}
\|f\|^{2}\leq \frac{1}{\lambda_{1}}\|\nabla f\|^{2}.
\end{equation}
\end{col}
Note also, that this inequality immediately follows from Lemma \ref{grad-laplace} and from the fact that  the norm of the operator $L^{-1/2}$  on the subspace of all functions which are orthogonal to $\varphi_{0}$ is $1/\sqrt{\lambda_{1}}$. 

In another  particular case when $\psi= \chi_{G}$ and
$$
f_{G}=\frac{1}{|G|}\sum_{v\in V(G)}f(v),
$$
one has that $f-f_{G}\chi_{G}$ belongs to the kernel of the corresponding functional $\Psi$ and it gives the next Corollary.

\begin{col} For every finite graph $G$ and for every $f\in \ell^{2}(G)$ the following holds
$$
\|f-f_{G}\chi_{G}\|^{2}\leq \frac{1}{\lambda_{1}} \|\nabla f\|^{2}.
$$
\end{col}

 \begin{thm} \label{Poinc}
Let $G$ be a finite graph and $\Psi$ be a functional on  $\ell^{2}(G)$ such that  $\Psi(\chi_{G})$ is not zero.  Then the following Poincare inequality holds for every $f\in \ell^{2}(G)$ and every $\epsilon>0$
 \begin{equation}\label{one-set}
\|f\|^{2}\leq \frac{\theta}{\lambda_{1}}        (1+\epsilon)     \|\nabla f\|^{2}     +
  \frac{1+\epsilon}{\epsilon}\frac{|G|^{2}}{|\Psi(\chi_{G})|^{2}}|\Psi(f)|^{2},\>\>\>\>f\in \ell^{2}(G),\>\>\epsilon>0,
 \end{equation}
where $\theta$ is defined in (\ref{Teta}). 
 \end{thm}
\begin{proof}
One has 
$$
\|f\|^{2}\leq \left \|\left(f-\frac{\Psi(f)}{\Psi(\chi_{G})}\chi_{G}\right)+ \frac{\Psi(f)}{\Psi(\chi_{G})}\chi_{G}\right\|^{2}
$$
Next,  we  apply  the  inequality 
\begin{equation}\label{algebra}
|A|^{2}\leq        (1+\epsilon)           \left|A-B\right|^{2}+  \frac{1+\epsilon}{\epsilon}\left|B\right|^{2},
\end{equation}
which holds for every  positive $\epsilon>0$. This inequality follows from two obvious inequalities
$$
|A|^{2}\leq |A-B|^{2}+2|B||A-B|+|B|^{2}
$$
 and
 $$
 2|B||A-B|\leq \epsilon|A-B|^{2}+\epsilon^{-1}|B|^{2},\>\>\>\epsilon>0.
 $$
Choosing an $\epsilon>0$ and using inequality   (\ref{algebra}) one obtains  
\begin{equation}\label{line1}
\|f\|^{2}\leq (1+\epsilon) \left \|f-\frac{\Psi(f)}{\Psi(\chi_{G})}\chi_{G}\right \|^{2}+  \frac{1+\epsilon}{\epsilon}\frac{|G|^{2}}{|\Psi(\chi_{G})|^{2}}|\Psi(f)|^{2}.
\end{equation}
Now an application of Corollary \ref{weighted-Poinc} gives the result.
Theorem is proved.
\end{proof}
In the case  when $\Psi$ is defined by $\psi=\chi_{G}$ one has that
$$
\frac{\Psi(f)}{\Psi(\chi_{G})}\chi_{G}=f_{G}\chi_{G},\>\>\>\>f_{G}=\frac{1}{|G|}\sum_{v\in V(G)}f(v).
$$
Since in this case $\theta$ in (\ref{Teta}) is $1$, $|G|^{2}/|\Psi(\chi_{G})|^{2}=1$,  and $\Psi(f)=\sum_{v\in V(G)}f(v)$ we obtain
 \begin{col} \label{Poinc-1}
 For every  connected and finite graph $G$   which contains more than one vertex the following Poincar\'e inequality holds
 \begin{equation}
\|f\|^{2}\leq 
(1+\epsilon)\frac{1}{\lambda_{1}}\|\nabla f\|^{2}+\frac{1+\epsilon}{\epsilon}\left|\sum_{v\in V(G)}f(v)\right|^{2},\>\>\>\>f\in \ell^{2}(G),\>\>\epsilon>0.
 \end{equation}
 \end{col}

\section{A generalized  Poincare-type inequality for finite and infinite graphs}\label{sec 3}

Let $G$ be a \textit{ finite or infinite} and countable connected graph and $S\subset V(G)$ is a finite and connected subset of vertices which we will treat as an \textit{ induced} graph  and will  denote by the same letter $S$. 
We remind that this  means that the set of vertices of such graph, which will be denoted as $V(S)$, is exactly the set of vertices in $S$ and the set of edges is the set of  edges in $E(G)$ whose both ends belong to $S$.
Let $L_{S}$ and $\|\nabla_{S}\left(f|_{S}\right)\|$ be the Laplace operator and the weighted gradient constructed according to (\ref{L})  and (\ref {gr}) for  the \textit{induced graph $S$}. Let $w_{S}(u,v),\>\>u,v\in V( S),$ and 
$$
w_{S}(v)= \sum_{u \in V(S)} w_{S}(u,v),\>\>v\in V(S),
$$ 
be the corresponding weight functions.
We notice that for every induced subgraph $S$ one has the inequalities
and every $u, v\in V(S)$ one has $w(u,v)=w_{S}(u,v)$. However, in general $w(u)\geq w_{S}(u)$.

Below we consider a cover of $V(G)$ by finite and connected sets of vertices  $S_{j},\>j\in J.$ 
We are using the same assumptions and notations which were introduced in {\bf Assumptions 1} in Introduction.

\begin{thm}\label{Main-Th}
Let $G$ be a connected finite or infinite and countable graph. Suppose that (\ref{cover}) holds true.  Let $L_{j}$  be the Laplace operator of the  induced  subgraph $S_{j}$ whose first nonzero eigenvalue is $\lambda_{1,j}$. The following inequality holds for every $f\in \ell^{2}(G)$ and every $\epsilon>0$

\begin{equation}\label{main-ineq-0}
\|f\|^{2}\leq    (1+\epsilon)  \sum_{j\in J}\frac{\theta_{j}}{\lambda_{1,j}}         \|\nabla_{j} f_{j}\|^{2}    +
  \frac{1+\epsilon}{\epsilon}\sum_{j\in J}\frac{|S_{j}|^{2}}{|\Psi_{j}(\chi_{j})|^{2}}|\Psi_{j}(f_{j})|^{2},
\end{equation}
where $\Psi_{j}(f)= \langle f, \psi_{j} \rangle$, function $\psi_{j}\in  \ell^{2}(G)$ has support in $S_{j}$, 
$$
\Psi_{j}(\chi_{j})=\sum_{v\in S_{j}} \psi_{j}(v)\neq 0,
$$
and
\begin{equation}\label{teta}
 \theta_{j}=\frac{|S_{j}|\|\psi_{j}\|^{2}}{|\Psi_{j}(\chi_{j} )|^{2}}.
\end{equation}
\end{thm}

\begin{proof}
One has
\begin{equation}
\|f\|^{2}=\sum_{v\in V(G)}|f(v)|^{2}=\sum_{j\in J}\left(\sum_{v\in V( S_{j})}|f_{j}(v)|^{2}\right).
\end{equation}
We  apply  Theorem \ref{Poinc} to have for every $j\in J$ and every $\epsilon>0,$
 \begin{equation}
\sum_{v\in V( S_{j})}|f_{j}(v)|^{2}\leq    
(1+\epsilon) \frac{\theta_{j}}{\lambda_{1,j}}         \|\nabla_{j} f_{j}\|^{2}     +
  \frac{1+\epsilon}{\epsilon}\frac{|S_{j}|^{2}}{|\Psi_{j}(\chi_{j})|^{2}}|\Psi_{j}(f_{j})|^{2},
 \end{equation}
and then we have for $f\in \ell^{2}(G),\>\>\epsilon>0,$

\begin{equation}
\|f\|^{2}\leq    (1+\epsilon)  \sum_{j\in J}  \frac{\theta_{j} }{\lambda_{1,j}}       \|\nabla_{j} f_{j}\|^{2}     +
  \frac{1+\epsilon}{\epsilon}\sum_{j\in J}\frac{|S_{j}|^{2}}{|\Psi_{j}(\chi_{j})|^{2}}|\Psi_{j}(f_{j})|^{2}.
\end{equation}
 Theorem is proved.
\end{proof}

As a consequence we obtain the following.

\begin{thm}\label{Main-Th-1}
If in addition to assumptions of Theorem \ref{Main-Th} we have that 
\begin{equation}\label{Theta}
\Theta_{\Xi}
=\sup_{j\in J} \theta_{j}<\infty,\>\>\>\>\>\>\Xi=\left(\{S_{j}\}_{j\in J},\>\{\Psi_{j}\}_{j\in J}\right),
\end{equation}
where $\theta_{j}$ is  computed according to (\ref{Teta}) and 
\begin{equation}\label{Lambda}
\Lambda_{\mathcal{S}}
=\inf_{j\in J}\lambda_{1,j}>0,\>\>\>\>\>\>\mathcal{S}=\{S_{j}\}_{j\in J},
\end{equation}
then  the following inequality  holds for every $f\in \ell^{2}(G)$ and every $\epsilon>0$
\begin{equation}\label{main-ineq}
\|f\|^{2}\leq (1+\epsilon)\frac{\Theta_{\Xi}
}{\Lambda_{\mathcal{S}}
} \|L^{1/2}f\|^{2}+\frac{1+\epsilon}{\epsilon}\sum_{j\in J}\frac{|S_{j}|^{2}}{|\Psi_{j}(\chi_{j})|^{2}}\left|\Psi_{j}(f_{j})\right|^{2}.
\end{equation}

\end{thm}
Proof of this statement follows from Theorem \ref{Main-Th} and 
Lemma \ref{grad-laplace} according to which 
$$
 \sum_{j\in J}      \|\nabla_{j} f_{j}\|^{2} \leq   \sum_{j\in J}      \|\nabla_{j} f_{j}\|^{2}   \leq \|\nabla f\|^{2}=\|L^{1/2} f\|.
 $$
Let's consider a few interesting cases. 

\begin{col} \label{noise}If all the notations and conditions of  Theorems \ref{Main-Th} and   \ref{Main-Th-1}  are satisfied and if for every $j$ the corresponding function $\psi_{j}=\chi_{j}$ is the characteristic function of a subset of vertices $U_{j}\subseteq S_{j}$ then
 the following inequalities hold
\begin{equation}
\|f\|^{2}\leq    (1+\epsilon)  \sum_{j\in J}  \frac{|S_{j}|}{\lambda_{1,j}|U_{j}|}       \|\nabla_{j} f_{j}\|^{2}     +
  \frac{1+\epsilon}{\epsilon}     \left(  \sum_{j\in J} \frac{|S_{j}|^{2}}{|    U_{j}|^{2}}\right)    \left|\sum_{v\in U_{j}}  f(v)\right|^{2},
\end{equation}
and
\begin{equation}
\|f\|^{2}\leq    (1+\epsilon)  \frac{1}{\Lambda_{\mathcal{S}}}  \sup_{j\in J}\frac{|S_{j}| }{|U_{j}|}
    \|L^{1/2} f\|^{2} +
  \frac{1+\epsilon}{\epsilon} \left(  \sum_{j\in J} \frac{|S_{j}|^{2}}{|    U_{j}|^{2}}\right)    \left|\sum_{v\in U_{j}}  f(v)\right|^{2}.
\end{equation}
In particular, if $U_{j}=S_{j}$ for every $j\in J$ then
\begin{equation}
\|f\|^{2}\leq    (1+\epsilon)  \sum_{j\in J}  \frac{1}{\lambda_{1,j}}       \|\nabla_{j} f_{j}\|^{2}     +
  \frac{1+\epsilon}{\epsilon}    \sum_{j\in J}   \left|\sum_{v\in S_{j}}  f(v)\right|^{2},
\end{equation}                                                
 and                                    
\begin{equation}
\|f\|^{2}\leq    (1+\epsilon)  \frac{1 }{\Lambda_{\mathcal{S}}
}     \|L^{1/2} f\|^{2} +  \frac{1+\epsilon}{\epsilon} \sum_{j\in J}\left|\sum_{v\in S_{j}}  f(v)\right|^{2}.
\end{equation}

\end{col}
Indeed, it follows from the fact that in this situation $\|\psi_{j}\|^{2}=|U_{j}|,\>\>\>$ $|\Psi_{j}(\chi_{j})|^{2}=|U_{j}|^{2}$ and
$$
 \theta_{j}=\frac{|S_{j}|\|\psi_{j}\|^{2}}{|\Psi_{j}(\chi_{j} )|^{2}}=\frac{|S_{j}|}{|U_{j}|}.
$$
The condition (\ref{Theta}) boils down to $\sup_{j\in J}|S_{j}|<\infty$.

\bigskip

\begin{col}
Suppose that all the  notations and conditions of  Theorems \ref{Main-Th} and   \ref{Main-Th-1} are satisfied. If for every $j$ the corresponding function $\psi_{j}$ is a Dirac measure $\delta_{v_{j}}$ at a vertex $v_{j}\in S_{j}$ then
\begin{equation}\label{local 0s}
\|f\|^{2}\leq    (1+\epsilon)  \sum_{j\in J}  \frac{|S_{j}|}{\lambda_{1,j}}       \|\nabla_{j} f_{j}\|^{2}     +
  \frac{1+\epsilon}{\epsilon}\sum_{j\in J}|S_{j}|^{2}|  |f(v_{j})|^{2},
\end{equation}
and
\begin{equation}\label{global 0s}
\|f\|^{2}\leq    (1+\epsilon)  \frac{\sup_{j\in J}|S_{j}| }{\Lambda_{\mathcal{S}}
}     \| L^{1/2} f\|^{2} +  \frac{1+\epsilon}{\epsilon}\sum_{j\in J}|S_{j}|^{2}|  |f(v_{j})|^{2}.
\end{equation}
\end{col}
\begin{proof}
In this case one has $ \|\psi_{j}\|=1, \>\>\Psi_{j}(f)=f(v_{j}), \>\>\Psi_{j}(\chi_{j})=1,\>\>\theta_{j}=|S_{j}|$ for every $j\in J$.

\end{proof}
The next corollary is about functions which annihilate all the functionals $\Psi_{j},\>j\in J$. 
\begin{col}\label{zeros}
If all the   notations and conditions of  Theorems \ref{Main-Th} and   \ref{Main-Th-1}  are satisfied and  for a function $f\in \ell^{2}(G)$ one has that  
$$
f\in \bigcap_{j\in J} Ker \>\Psi_{j}
$$  
 then
\begin{equation}\label{Many zeros-1}
\|f\|^{2}\leq    \sum_{j\in J}  \frac{\theta_{j} }{\lambda_{1,j}}       \|\nabla_{j} f_{j}\|^{2}  ,\>\>\>\>\>\>\>f\in \bigcap_{j\in J} Ker \Psi_{j},
\end{equation}
and 
\begin{equation}\label{Many zeros-2}
\|f\|^{2}\leq     \frac{\Theta_{\Xi}
 }{\Lambda_{\mathcal{S}}
}    \|L^{1/2} f\|^{2},\>\>\>\>\>\>f\in \bigcap_{j\in J} Ker \Psi_{j}.   
\end{equation}
\end{col}

\begin{rem} If $J_{0}\subset J$ and $G_{0}=\cup_{j\in J_{0}} G_{j}$ then every inequality in this section can be replaced by a similar one in which the term $\|f\|^{2}$ on the left is replaced  by 
$$
\|f\|^{2}_{G_{0}}=\sum _{v\in G_{0}}\|f\|^{2},
$$
and summation over $J$ on the right is replaced by summation over $J_{0}$. 
For example, the last two inequalities (\ref{Many zeros-1}) and (\ref{Many zeros-2}) would take the form 
\begin{equation}\label{Many zeros-3}
\|f\|^{2}_{G_{0}}\leq     \sum_{j\in J_{0}}  \frac{\theta_{j} }{\lambda_{1,j}}       \|\nabla_{j} f_{j}\|^{2}  ,\>\>\>\>\>\>\>f\in \bigcap_{j\in J_{0}} Ker \Psi_{j},
\end{equation}
and 
\begin{equation}\label{Many zeros-4}
\|f\|^{2}_{G_{0}}\leq     \frac{\Theta_{\Xi}
 }{\Lambda_{\mathcal{S}}
}    \|L_{G_{0}}^{1/2} f_{0}\|^{2},\>\>\>\>f_{0}=f|_{G_{0}},\>\>\>f_{0}\in \bigcap_{j\in J_{0}} Ker \Psi_{j},   
\end{equation}
where $L_{G_{0}}$ is the Laplacian of the induced graph $G_{0}$. 

Note, that in the case when $\{\Psi_{j}\}$ is a set of "uniformly" distributed Dirac functions the last inequality (\ref{Many zeros-4}) is called sometimes "the inequality for functions with many zeros".

\end{rem}

\section{Paley-Wiener vectors in $\ell^{2}(G)
$}\label{PWspace}

Our next goal is to introduce the so-called Paley-Wiener functions (bandlimited functions) for which a sampling theory will be developed in the setting of combinatorial graphs. We use for this the self-adjoint positive definite operator
$L$ in a Hilbert space $\ell^{2}(G)$.
In the case when $L$ has discrete spectrum (which is always the case with finite graphs) then the Paley-Wiener space $PW_{\omega}(L)$ is simply the span of eigenfunctions of $L$ whose corresponding eigenvalues are not greater $\omega$. However, when graph is infinite and spectrum of $L$ is continuous it takes a bigger effort to define spaces $PW_{\omega}(L)$.

Consider a self-adjoint positive definite operator
$L$ in a Hilbert space $\ell^{2}(G)
$.
According to the spectral theory \cite{BS} 
for self-adjoint non-negative  operators 
there exists a direct integral of
Hilbert spaces $\mathcal{H}=\int \mathcal{H}(\lambda )dm (\lambda )$ and a unitary
operator $\mathcal{F}$ from $\ell^{2}(G)
$ onto $\mathcal{H}$, which
transforms the domains of $L^{k}, k\in \mathbf{N},$
onto the sets
$\mathcal{H}_{k}=\{x \in \mathcal{H}|\lambda ^{k}x\in\mathcal{H} \}$
with the norm 
\begin{equation}\label{FT}
\|x(\lambda)\|_{\mathcal{H}_{k}}= \left<x(\lambda),x(\lambda)\right>^{1/2}_{\mathcal{H}(\lambda)}=
$$
$$
\left (\int^{\infty}_{0}
 \lambda^{2k}\|x(\lambda )\|^{2}_{\mathcal{H}(\lambda )} dm
 (\lambda ) \right )^{1/2}.
 \end{equation}
and satisfies the identity
$\mathcal{F}(L^{k} f)(\lambda)=
 \lambda ^{k} (\mathcal{F}f)(\lambda), $ if $f$ belongs to the domain of
 $L^{k}$.
 We call the operator $\mathcal{F}$ the Spectral Fourier Transform. As known, $\mathcal{H}$ is the set of all $m $-measurable
  functions $\lambda \mapsto x(\lambda )\in \mathcal{H}(\lambda ) $,
  for which the following norm is finite:
$$
\|x\|_{\mathcal{H}}=
\left(\int ^{\infty }_{0}\|x(\lambda )\|^{2}_{\mathcal{H}(\lambda )}dm
(\lambda ) \right)^{1/2} 
$$
For the characteristic function ${\bf 1}_{[0,\>\omega]}$ one can introduce the projector ${\bf 1}_{[0,\>\omega]}(L)$ by using the formula
\begin{equation}\label{Op-function}
{\bf 1}_{[0,\>\omega]}(L) f=\mathcal{F}^{-1}{\bf 1}_{[0,\>\omega]}( \lambda)\mathcal{F}f,\>\>\>f\in \mathcal{H}.
\end{equation}

\begin{defn}\label{PW}
 The Paley-Wiener space $PW_{\omega} (L)\subset \Ltwo(G)$ is defined as  the image space of the projection operator ${\bf 1}_{[0,\>\omega]}(L)$. 
 \end{defn}
Many properties of Paley-Wiener spaces for general self-adjoint operators in Hilbert spaces can be found in our papers 
\cite{Pes00}. The most important for us is the following.

\begin{thm}A function $f\in \ell^{2}(G)$ belongs to the spaces $PW_{\omega}(L)$ if and only if 
  the following Bernstein inequalities  holds true
\begin{equation}\label{Bern}
\|L^{s}f\|\leq \omega^{s}\|f\| \quad \mbox{for all} \, \,  s\in \mathbf{R}_{+};
\end{equation}
\end{thm}

\section{A sampling theorem and a reconstruction methods using frames}\label{frames}

\subsection{A sampling theorem}
Let's remind that 
a set of vectors $\{\xi_{\nu}\}$  in a Hilbert space $H$ is called a Hilbert  frame if there exist constants $A, B>0$ (frame bounds) such that for all $f\in H$ 
\begin{equation}\label{Frame ineq}
A\|f\|^{2}\leq \sum_{\nu}\left|\left<f,\xi_{\nu}\right>\right|^{2}     \leq B\|f\|^{2}.
\end{equation}
What is remarkable about frames is the fact that one can perfectly reconstruct a vector $f$ from its projections $\left<f,\xi_{\nu}\right>$. Namely, 
according to the general theory of Hilbert frames \cite{DS}, \cite{Gr} the frame  inequality (\ref{Frame ineq})  implies that there exists a dual frame $\{\Omega_{\nu}\}$ (which is not unique in general) 
for which the following reconstruction formula holds 
\begin{equation}\label{dual}
f=\sum_{v}\left<f,\xi_{\nu}\right>\Omega_{\nu}.
\end{equation}
In general  it is not easy to find a dual frame. For this reason one can resort 
to the following  frame algorithm (see  \cite{Gr}, Ch. 5) which performs reconstruction by iterations.       
Given a relaxation parameter $0<\rho<\frac{2}{B}$, set $\eta=\max\{|1-\rho A|,\> |1-\rho B|\}<1$. Let $f_{0}=0$ and define recursively
\begin{equation}
\label{rec}
f_{n}=f_{n-1}+\rho \Phi(f-f_{n-1}),
\end{equation}
where $\Phi$ is the frame operator which is defined on $H$ by the formula
$
\Phi f=\sum_{\nu}\left<f,\xi_{\nu}\right>\xi_{\nu}.
$
In particular, $f_{1}=\rho \Phi f=\rho\sum_{j}\left<f, \xi_{\nu}\right> \xi_{\nu}$. Then  $\lim_{n\rightarrow \infty}f_{n}=f$ with a geometric rate of convergence, that is, 
\begin{equation}
\label{conv}
\|f-f_{n}\|\leq \eta^{n}\|f\|.
\end{equation}
In particular,  for the choice $\rho=\frac{2}{A+B}$ the convergence factor is 
 $$ 
\eta=\frac{B-A}{A+B}.
$$
Let $\delta_{s_{i}},\>i\in I$ be the Dirac delta concentrated at the vertex $s_{i}$.
\begin{thm}\label{Main-Th2}
If all the  notations and conditions of  Theorems \ref{Main-Th} and   \ref{Main-Th-1}  hold then the set of functionals $\{\Psi_{j}\}_{j\in J}$ is a frame in any space $PW_{\omega}(L)$ as long as
\begin{enumerate}

\item 
 \begin{equation}\label{new-0}
 0<\omega<\frac{\Lambda_{\mathcal{S}}}{(1+\epsilon)\Theta_{\Xi}},\>\>\>\epsilon>0,
 \end{equation}

\item there exists a constant $c=c(\{S_{j}\}, \{\Psi_{j}\})$ such that for every $j\in J$ the following inequality holds
\begin{equation}\label{new-1}
\frac{|S_{j}|^{2}}{|\Psi_{j}(\chi_{j})|^{2}}\leq c,
\end{equation}
\item there exists a constant $C=C( \{S_{j}\}, \{\Psi_{j}\}   )$ such that for every $j\in J$ one has 
\begin{equation}\label{new-2}
\|\psi_{j}\|^{2}\leq C.
\end{equation}

\end{enumerate}
In other words, if for an $\epsilon>0$ the following inequality holds
 
\begin{equation}\label{gamma}
\gamma=(1+\epsilon)\frac{\Theta_{\Xi}}{\Lambda_{\mathcal{S}}}\omega
<1,\>\>\epsilon>0,
\end{equation}
along with (\ref{new-1}) and (\ref{new-2}) then
\begin{equation}\label{P-P}
\frac{(1-\gamma)\epsilon}{(1+\epsilon)c}\|f\|^{2}\leq \sum_{j\in J}\left|\Psi_{j}(f)\right|^{2}\leq C\|f\|^{2}.
\end{equation}
\end{thm}
\begin{proof}
We notice that since support of $\psi_{j}$ is in $S_{j}$ we have
$$
\Psi_{j}(f_{j})=\left<f,\psi_{j}\right>=\Psi_{j}(f).
$$
Now, if $f\in PW_{\omega}(L)$ then by the Bernstein inequality (\ref{Bern}) the (\ref{main-ineq}) can be rewritten as 
$$
\|f\|^{2}\leq(1+\epsilon) \frac{\Theta_{\Xi}}{\Lambda_{\mathcal{S}}}\omega
 \| f\|^{2}+\frac{1+\epsilon}{\epsilon}\sum_{j\in J}\frac{|S_{j}|^{2}}{|\Psi_{j}(\chi_{j})|^{2}}\left| \Psi_{j}(f_{j})\right|^{2}.
 $$
If (\ref{new-1}) and (\ref{gamma}) hold then one obtains the left-hand side of (\ref{P-P}).
On the other hand, we have
$$
\sum_{j\in J} |\Psi_{j}(f)|^{2}=\sum_{j\in J}\left|\sum_{v\in S_{j}}f_{j}(v)\psi_{j}(v)\right|^{2}
\leq
 \sum_{j\in J}\|\psi_{j}\|^{2}\|f_{j}\|^{2}\leq C\|f\|^{2}.
$$

Theorem is proven.
\end{proof}

Note, that for the classical Paley-Wiener spaces on the real line the inequalities similar to (\ref{P-P}) in the case when $\{\psi_{j}\}$ are delta functions were proved by Plancherel and Polya. Today they are better  known as the frame inequalities.
Now we can formulate sampling theorem based on average values.

	\begin{thm} \label{conclusion} Under the same conditions and notations as in Theorem \ref{Main-Th2} every function $f\in PW_{\omega}(L)$ is uniquely determined by the set of numbers   
	$
	\{\left<f, \psi_{j}\right>\}_{j\in J}
	$
 and can be reconstructed from this set of values in a stable way using dual frames (\ref{dual}) or the iterative frame algorithm (\ref{rec}). 

\end{thm}

\subsection{Important particular cases}\label{cases}

\begin{enumerate}

\item  (Sampling by averages-I).  If for every $j$ the corresponding function $\psi_{j}=\chi_{j}$ is the characteristic function of a subset of vertices $U_{j}\subset S_{j}$ then inequalities (\ref{new-0})-(\ref{new-2}) take the form respectively
$$
 0<\omega<\frac{\Lambda_{\mathcal{S}}}{(1+\epsilon)\sup_{j\in J}|S_{j}|},\>\>\>\>\>\frac{|S_{j}|^{2}}{|U_{j}|^{2}}\leq c, \>\>\>\>\> |U_{j}|\leq C,
$$
and the Plancherel-Polya inequalities (\ref{P-P}) hold with the same constants $c$ and $C$.
In particular, if $U_{j}=S_{j}$ for every $j\in J$ then (\ref{new-0}) takes the form
\begin{equation}\label{100}
 0<\omega<\frac{\Lambda_{\mathcal{S}}}{(1+\epsilon)\sup_{j\in J}|S_{j}|},
\end{equation}
the condition (\ref{new-1}) is trivially satisfied with $c=1$, and (\ref{new-2}) becomes $|S_{j}|\leq C$. The (\ref{P-P}) holds true with the corresponding constants $C$ and $c=1.$

\item (Sampling by averages-II). In the case $U_{j}=S_{j}$ and 
$$
\psi_{j}=\frac{1}{\sqrt{ |S_{j}|}}\chi_{j},
$$
 every $\theta_{j}$ in (\ref{teta}) is one and it gives that $\Theta_{\Xi}$ in (\ref{Theta}) is also one. Thus (\ref{new-0}) takes the form
\begin{equation}
 0<\omega<\frac{\Lambda_{\mathcal{S}}}{1+\epsilon},\>\>\>\epsilon>0.
\end{equation}
Moreover, in this case $C=c=1$. 
After all the Plancherel-Polya inequality (\ref{P-P}) becomes 
\begin{equation}\label{P-P-1}
\frac{(1-\gamma)\epsilon}{(1+\epsilon)}\|f\|^{2}\leq \sum_{j\in J}\left|\Psi_{j}(f)\right|^{2}\leq \|f\|^{2}, \>\>\>f\in PW_{\omega}(G),
\end{equation}
where
\begin{equation}\label{gamma-1}
\gamma=\frac{1+\epsilon}{\Lambda_{\mathcal{S}}}\omega
<1,\>\>\>\>\epsilon>0.
\end{equation}
\item (Point wise sampling). If for every $j$ the corresponding function $\psi_{j}$ is a Dirac measure $\delta_{v_{j}}$ at a vertex $v_{j}\in S_{j}$ then the condition (\ref{new-0}) takes the form (\ref{100}), the condition (\ref{new-1}) will have form $|S_{j}|^{2}\leq c$, the condition (\ref{new-2}) is trivially satisfied with $C=1$. The (\ref{P-P}) holds true with these constants.

\end{enumerate}	

\subsection{Reconstruction algorithms in terms of  frames}

	What we just proved in the previous section is that 
	under the same assumptions as above the set of functionals $f\rightarrow \left<f, \psi_{j}\right>$ is a frame in the subspace $PW_{\omega}(L)$.
	This fact allows to apply the well known result  of Duffin and Schaeffer
\cite{DS}  which describes  a stable method of
reconstruction of a function $f\in PW_{\omega}(L)$ from a
 set of samples $\{\left<f,\psi_{j}\right>\}$.

\begin{thm} If all the conditions of Theorem \ref{Main-Th} are satisfied
 then there exists a dual frame $\{\Omega_{j}\}$ in $ PW_{\omega}(L)$
 such that 
 $$
 f=\sum_{j}\left<f,\psi_{j}\right>\Omega_{j}=\sum_{j}\left<f,\Omega_{j}\right>\mathcal{P}_{\omega}\psi_{j}
 $$
where $\mathcal{P}_{\omega}$ is the orthogonal projection of $\ell^{2}(G)$ onto $PW_{\omega}(L)$.
\end{thm}

Another possibility for reconstruction is to use  frame algorithm (see section \ref{frames}). 

\section{Weighted Average Variational Splines and a reconstruction algorithm}\label{w-splines}

\subsection{Variational interpolating splines}
As in the previous sections we assume that $G$ is a connected finite or infinite  graph, $\mathcal{S}=\{S_{j}\}_{j\in J}$, is  a disjoint cover of $V(G)$ by connected and finite subgraphs $S_{j}$ and every $\psi_{j}\in \ell^{2}(S_{j}),\>\>\>j\in J,$ has support in $S_{j}$.

  For a  given sequence $\mathbf{a}=\{a_{j}\}\in l_{2}$ the set of all functions 
in  $\ell^{2}(G)$ such
 that $\Psi_{j}(f)=\left<f, \psi_{j}\right>=a_{j}$ will be denoted by $Z_{\mathbf{a}}$. In particular, $$
 Z_{\mathbf{0}}=\bigcap_{j\in J} Ker(\Psi_{j})
 $$
  corresponds to the sequence of zeros. 
 We consider the following optimization problem:

 \textit{For a given sequence $\mathbf{a}=\{a_{j}\}\in l_{2}$ find a function $f$ in
the set $Z_{\mathbf{a}}\subset \ell^{2}(G)$  which minimizes the functional
$$
u\rightarrow \|L^{k/2}u\|, \>\>\>\>u\in Z_{\mathbf{a}}.
$$}
\begin{thm}
Under the above assumptions the optimization problem has a unique 
solution for every $k$.
\end{thm}
\begin{proof}Using Theorem \ref{Main-Th} one can justify the following algorithm (see \cite{Pes98a}, \cite{Pes01}):
\begin{enumerate}
 
 \item Pick any function $f\in Z_{\mathbf{a}}$. 
 
\item  Construct  $\mathcal{P}_{0}f$ where $\mathcal{P}_{0}$ 
is  the orthogonal projection of $f$  onto
 $Z_{\mathbf{0}}$ with respect to the inner product
 $$
 \left<f,g\right>_{k}= \sum_{j}\left<f, \psi_{j}\right>\left<g, \psi_{j}\right>+ \langle
L^{k/2}f,
L^{k/2}g \rangle.
$$
\item 
The function $f-\mathcal{P}_{0}f$ is the unique solution to the given optimization
problem.
\end{enumerate}
\end{proof}

\begin{defn}
For $f\in \ell^{2}(G)$ the interpolating variational spline is denoted by $s_{k}(f)$  and it is the solution of the
minimization problem such that $s_{k}(f)-f\in Z_{\mathbf{0}}.$
\end{defn}
Clearly, "interpolation" is understood in the sense that
\begin{equation}\label{interp}
\Psi_{j}(s_{k}(f))=\Psi_{j}(f).
\end{equation}
One can easily prove  the following characterization of  variational splines.
\begin{thm}
A function $u\in \ell^{2}(G)$ is a variational spline  
if and only if $L^{k}u$ is orthogonal to $L^{k}Z_{\mathbf{0}}$.
\end{thm}

\subsection{Reconstruction using splines}

The following Lemma was proved in \cite{Pes98a}, \cite{Pes01}.

\begin{lem}\label{lemma}
If $A$ is a self-adjoint non-negative operator in a
Hilbert space $X$ and for an  $\varphi\in X$ and a positive $a>$
the following inequality holds true
$$
\|\varphi\|\leq a\|A\varphi\|,
$$ then for the same $\varphi \in H$, and all $ k=2^{l}, l=0,1,2,...$ the following
inequality holds
$$
\|\varphi\|\leq a^{k}\|A^{k}\varphi\|.
$$
\end{lem}

By using the same reasoning as in \cite{Pes98a}, \cite{Pes01} one can prove the following  reconstruction theorem. Below we are keeping  notations of Theorem \ref{Main-Th2}.

\begin{thm}  Let's assume that $G$ is a connected finite or infinite graph, $\{S_{j}\}_{j\in J}$  is a disjoint cover of $V (G)$ by connected and finite subgraphs $S_j$  and every $\psi_{j} \in \ell^{2}(S_{j}),\> j \in J$, has support in $S_j$. If 
\begin{equation}\label{omega-last}
0<\omega<\frac{\Lambda_{\mathcal{S}}}{\Theta_{\Xi}},
\end{equation}
\begin{equation}
\Theta_{\Xi}
=\sup_{j\in J} \theta_{j}= \theta_{j}=\frac{|S_{j}|\|\psi_{j}\|^{2}}{|\Psi_{j}(\chi_{j} )|^{2}},\>\>\>
\end{equation}
\begin{equation}
\Lambda_{\mathcal{S}}
=\inf_{j\in J}\lambda_{1,j},
\end{equation}
then any function $f$ in $PW_{\omega}(L),\>\>\>\> \omega>0,$ can be reconstructed from a set of values  $\{\left<f, \psi_{j}\right>\}$ using the
formula  
$$
f=\lim_{k\rightarrow\infty}s_{k}(f),\>\>\>\>k=2^{l},\>\>\>
l=0,1, ...,
$$
and the error estimate is
\begin{equation}\label{error}
\|f-s_{k}(f)\|\leq   2 \gamma^{k}\|f\|, \>\>\>\>k=2^{l},\>\>\>
l=0,1, ... ,
\end{equation}
where 
$$
\gamma=\frac{\Theta_{\Xi}}{\Lambda_{\mathcal{S}}}\omega<1.
$$

\end{thm}

\begin{proof}
For a $k=2^{l},\>\>l=0,1,2,....$ apply  to the function $f-s_{k}(f)$ inequality
 (\ref{main-ineq}) for any $\epsilon>0$:
\begin{equation}\label{main-ineq-11}
\|f-s_{k}(f)\|^{2}\leq (1+\epsilon)\frac{\Theta_{\Xi}
}{\Lambda_{\mathcal{S}}
} \|L^{1/2}(f-s_{k}(f))\|^{2}+
$$
$$
\frac{1+\epsilon}{\epsilon}\sum_{j\in J}\frac{|S_{j}|^{2}}{|\Psi_{j}(\chi_{j})|^{2}}\left|\Psi_{j}(f_{j})\right|^{2}.
\end{equation}
Since $s_{k}(f)$ interpolates $f$ the last term here is zero. Because $\epsilon$  here is any positive number it brings us to the next inequality
$$
\left\|f-s_{k}(f)\right\|^{2}\leq  \frac{\Theta_{\Xi}}{\Lambda_{\mathcal{S}}} \|L^{1/2}(f-s_{k}(f))\|^{2},
$$
and an application of  Lemma \ref{lemma} gives
$$
\left\|f-s_{k}(f)\right\|^{2}\leq \left(\frac{\Theta_{\Xi}}{\Lambda_{\mathcal{S}}}\right)^{k} \|L^{k/2}(f-s_{k}(f))\|^{2}.
$$
Using minimization property of $s_{k}(f)$ and the Bernstein inequality (\ref{Bern}) for $f\in PW_{\omega}(L)$ one obtains (\ref{error}). Theorem is proved.
\end{proof}
One can formulate similar statements adapted to particular cases listed in subsection \ref{cases}.

\section{Example. Average sampling on $\mathbb{Z}$}\label{example}

Let us  consider a one-dimensional infinite lattice $\mathbb{Z}=\{...,-1, 0, 1, ...\}$ as an unweighted  graph.  The dual
group of the commutative additive group $\mathbb{Z}$ is the
one-dimensional torus. The corresponding Fourier transform
$\mathcal{F}$ on the space $\ell^{2}(\mathbb{Z})$   is defined by the
formula
$$
\mathcal{F}(f)(\xi)=\sum_{k\in \mathbb{Z}}f(k)e^{ik\xi}, f\in
\ell^{2}(\mathbb{Z}), \xi\in [-\pi, \pi).
$$
It gives a unitary operator from $\ell^{2}(\mathbb{Z})$ on the space
$L_{2}(\mathbb{T})=L_{2}(\mathbb{T}, d\xi/2\pi),$ where
$\mathbb{T}$ is the one-dimensional torus and $d\xi/2\pi$ is the
normalized measure.  One can verify the following formula
$$
\mathcal{F}(Lf)(\xi)=4\sin^{2}\frac{\xi}{2}\mathcal{F}(f)(\xi).
$$
The next result is obvious.
\begin{thm} The spectrum of the Laplace operator $L$ on the one-dimensional
lattice $\mathbb{Z}$ is the interval  $[0, 4]$. A function $f$ belongs
to the space $PW_{\omega}(\mathbb{Z}), 0\leq \omega\leq 4,$ if and
only if the support of $\mathcal{F}f$ is a subset
 of $[-\pi,\pi)$ on which
$4\sin^{2}\frac{\xi}{2}\leq \omega$.

\end{thm}
We consider  the cover $\Xi=\{S_{j}\}$ of $\mathbb{Z}$ by disjoint sets $S_{j}=\{j-1, \>j,\>j+1\}$ where $j$ runs over all integers divisible by $3$: $\{..., -3, 0, 3, ... \}=3\mathbf{Z}$. 
 We treat every $S_{j}$ as an induced graph whose set of vertices is $V(S_{j})=\{j-1, \>j,\>j+1\},\>\>j\in 3\mathbf{Z},$ and which has two edges  $(j-1, \>j) $ and  $(j, \>j+1)$. Let' introduce functionals  $\Psi_{j}$ as 
 \begin{equation}\label{last-0}
\Psi_{j}(f)= \langle f,\psi_{j}\rangle=\frac{1}{\sqrt{3}}\left(f(j-1)+f(j)+f(j+1)\right),\>\>\>j\in 3\mathbf{Z,\>\>\>}f\in \ell^{2}(\mathbb{Z}).
\end{equation}
 One can check that  spectrum of the Laplace operator $L_{j}$ on $S_{j}$ defined by (\ref{L}) contains just  three values $\{0,\>2,\>4\}$. Thus $\Lambda_{\mathcal{S}}=2$.   For  an $0<\omega<4$ and $\epsilon>0$  condition (\ref{omega-last}) takes form
\begin{equation}\label{last}
 \gamma=(1+\epsilon)\frac{\omega}{2}<1.
\end{equation}
Note, that since  $1+\epsilon$ can be arbitrary  close to $1$  the condition (\ref{last}) implies that  $0<\omega<2$. 
As an application of Theorem \ref{conclusion} we obtain the following result.
\begin{thm} 
If  $0<\omega<2$  then every $f\in PW_{\omega}(\mathbb{Z})$ is uniquely determined by its average values $\{\langle f, \psi_{j}\rangle\}$ defined in (\ref{last-0}) and can be reconstructed from them in a stable way.

\end{thm}

  In particular, if instead of infinite graph $\mathbb{Z}$ one would consider  a path graph $\mathbb{Z}_{N}$ whose eigenvalues are given by formulas $2-2\cos\frac{k\pi}{N-1},\>\>\>k=0,1, ..., N-1, $ the last Theorem  would mean that any eigenfunction with eigenvalue from a lower half of the spectrum  is uniquely determined and can be reconstructed from averages (\ref{last-0}).

\end{document}